\documentclass[12pt]{article}

\usepackage{amsmath}
\usepackage{amsthm}
\usepackage{amsxtra}

\usepackage[T1]{fontenc}

\usepackage{amsfonts,amssymb}
\usepackage{latexsym}

\newtheorem{Th}{Theorem}
\newtheorem{Prop}[Th]{Proposition}
\newtheorem{Lm}[Th]{Lemma}

\theoremstyle{definition}

\newcommand{\supp}{\mathrm{supp}}

\newcommand{\id}{\mathrm{I}}

\newcommand{\ie}{\text{i.e.\;\,}}

\numberwithin{equation}{section}

\begin{document}
\title{On  spectra   of representations  and  graphs.  Erratum.}
\author{ {\bf Artem Dudko}  \\
                    IM PAN, Warsaw, Poland \\
          adudko@impan.pl \\
         {\bf Rostislav Grigorchuk
         } \\ Texas A\&M University, College Station, TX, USA  \\      grigorch@math.tamu.edu }
\maketitle
\date{}
\noindent Unfortunately the proof of the main result of \cite{DG15 Spectra}, Theorem 1, has a flaw. Namely, Lemma 13 used in the proof of Proposition 11 is correct only under an additional assumption that the operator $A$ is normal (adjoint for the one-sided shift operator in $l^2(\mathbb N)$ provides a counterexample). Below we prove a version of Lemma 13 that does not require the normality assumption and apply it to prove Proposition 11. In addition, the same version of the lemma appears in paper \cite{DG18 Shape} (as Lemma 3.1) where it is used in the proof of Theorem 1.6. We also explain here how to use the new version of Lemma 13 to correct the proof of Theorem 1.6 from \cite{DG18 Shape}.

\paragraph*{Corrections of \cite{DG15 Spectra}.} Let us briefly remind the notations.
For an action of a group $G$ on a space $X$ and a point $x\in X$ the notation $\rho_x$ stands for the quasi-regular representation of $G$ acting on $l^2(Gx)$. The \emph{orbital graph} $\Gamma_x=\Gamma_{x,g_1,\ldots,g_n}$ is a generalization of the Schreier graph for the case when a collection of elements $S=\{g_1,\ldots,g_n\}\subset G$ does not generate $G$. We refer the reader to \cite{DG15 Spectra} for precise definitions. For the reader's convenience recall the statement of Proposition 11 from \cite{DG15 Spectra}.
\begin{Prop}\label{Prop-equiv-M} Let $G$ act on a space $X$, $x,y\in X$ and $m\in\mathbb C[G]$. Let $\supp(m)=\{g_1,g_2,\ldots,g_n\}$. If the orbital graphs $\Gamma_{x,g_1,\ldots,g_n}$ and $\Gamma_{y,g_1,\ldots,g_n}$  are locally isomorphic then $\sigma(\rho_x(m))=\sigma(\rho_y(m))$ .
\end{Prop}
\noindent To prove Proposition \ref{Prop-equiv-M} we use the following version of Lemma 13 from \cite{DG15 Spectra}.
\begin{Lm}\label{LmEquivNorm} Let $A$ be any bounded nonzero linear operator on a Hilbert space and $R\geqslant 2\|A\|$. Then the following assertions are equivalent:
\begin{itemize} \item[$1)$] $\lambda\in \sigma(A)$,
\item[$2)$] $1\in \sigma(\mathrm{I}-
\tfrac{1}{R^2}(A-\lambda\mathrm{I})(A-\lambda\mathrm{I})^{*})\cup \sigma(\mathrm{I}-
\tfrac{1}{R^2}(A-\lambda\mathrm{I})^*(A-\lambda\mathrm{I})),$
\end{itemize}
where $\mathrm{I}$ is the identity operator.
\end{Lm}
\begin{proof} Let $1\notin\sigma(\mathrm{I}-
\tfrac{1}{R^2}(A-\lambda\mathrm{I})(A-\lambda\mathrm{I})^{*})\cup \sigma(\mathrm{I}-
\tfrac{1}{R^2}(A-\lambda\mathrm{I})^*(A-\lambda\mathrm{I}))$. Then both $(A-\lambda\mathrm{I})(A-\lambda\mathrm{I})^{*}$ and $(A-\lambda\mathrm{I})^*(A-\lambda\mathrm{I})$ are invertible operators. In particular, there exist bounded linear operators $B,C$ on the same Hilbert space such that $$(A-\lambda\mathrm{I})(A-\lambda\mathrm{I})^{*}B=C(A-\lambda\mathrm{I})^*(A-\lambda\mathrm{I})=\mathrm{I}.$$ This means that $A-\lambda\mathrm{I}$ has both a left and a right inverse. Obviously, these inverses coincide, and so $A-\lambda\mathrm{I}$ is invertible. Thus, $\lambda\notin \sigma(A)$.

On the other hand, if $\lambda\notin\sigma(A)$ then $A-\lambda\mathrm{I}$ and $(A-\lambda\mathrm{I})^{*}$ are invertible, therefore $\tfrac{1}{R^2}(A-\lambda\mathrm{I})(A-\lambda\mathrm{I})^{*}$ and $(A-\lambda\mathrm{I})^*(A-\lambda\mathrm{I})$ are invertible. It follows that $$1\notin \sigma(\mathrm{I}-
\tfrac{1}{R^2}(A-\lambda\mathrm{I})(A-\lambda\mathrm{I})^{*})\cup \sigma(\mathrm{I}-
\tfrac{1}{R^2}(A-\lambda\mathrm{I})^*(A-\lambda\mathrm{I})).$$
\end{proof}

\begin{proof}[Proof of Proposition \ref{Prop-equiv-M}.]
Let $G,X,x,y,m$ be as in the statement of Proposition \ref{Prop-equiv-M} and let the orbital graphs $\Gamma_x=\Gamma_{x,g_1,\ldots,g_n}$ and $\Gamma_y=\Gamma_{y,g_1,\ldots,g_n}$ be locally isomorphic.
Set $$R=2\sum\limits_{g\in\supp(m)} |m(g)|.$$
Fix a point $\alpha$ from $\sigma(\rho_x(m))$ and let us show that $\alpha\in \sigma(\rho_y(m))$.

Clearly,
$|\alpha|\leqslant\|\rho_x(m)\|\leqslant \tfrac{1}{2}R$. Using Lemma \ref{LmEquivNorm} we obtain that $$1\in \sigma(\mathrm{I}-
\tfrac{1}{R^2}(\rho_x(m)-\alpha\mathrm{I})(\rho_x(m)-\alpha\mathrm{I})^{*})\cup \sigma(\mathrm{I}-
\tfrac{1}{R^2}(\rho_x(m)-\alpha\mathrm{I})^{*}(\rho_x(m)-\alpha\mathrm{I})).$$ Assume, for instance, that $1\in \sigma(\mathrm{I}-
\tfrac{1}{R^2}(\rho_x(m)-\alpha\mathrm{I})(\rho_x(m)-\alpha\mathrm{I})^{*})$ (the second case can be treated similarly). The
operator $\mathrm{I}-
\tfrac{1}{R^2}(\rho_x(m)-\alpha\mathrm{I})(\rho_x(m)-\alpha\mathrm{I})^{*}$
is of the form $\rho_x(s)$ for some $s\in \mathbb C[G]$,
 positive and of norm $\leqslant 1$. In fact,
$\|\rho_x(s)\|=1$, since we have $1\in\sigma(\rho_x(s))$.

Further, consider orbital graphs $\Gamma_x$ and $\Gamma_y$. Let $\epsilon>0$. Since
$$\sup\limits_{\xi:\|\xi\|=1} (\rho_x(s)\xi,\xi)=1$$ we can find $l\in\mathbb N$ and a vector $\eta\in l^2(Gx)$
 supported on $B_l(x)$ such that $(\rho_x(s)\eta,\eta)>1-\epsilon$. Let
 $v\in\Gamma_y$ be such that $B_l(v)\subset \Gamma_y$ is isomorphic (as a rooted labeled graph) to
  $B_l(x)$. Let $\eta'\in l^2(B_l(v))\subset l^2(\Gamma_y)$ be a copy of $\eta$ via this isomorphism.
   Then one has:
  $$(\rho_x(s)\eta',\eta')=(\rho_y(s)\eta,\eta)>1-\epsilon.$$ Since $\epsilon>0$ is arbitrary it follows that $\|\rho_y(s)\|\geqslant 1$. By the choice of $R$ we have that the spectrum of
  $$\rho_y(s)=\mathrm{I}-
\tfrac{1}{R^2}(\rho_y(m)-\alpha\mathrm{I})(\rho_x(m)-\alpha\mathrm{I})^{*})$$ is a subset of $[0,1]$, and therefore $\|\rho_y(s)\|=1$. By positivity of $\rho_y(s)$ we have $1\in\rho_y(s)$. Using Lemma \ref{LmEquivNorm} we obtain that $\alpha\in\sigma(\rho_y(m))$ and
  this finishes the proof of Proposition \ref{Prop-equiv-M}. \end{proof}
\paragraph*{Corrections of \cite{DG18 Shape}.} Let us briefly recall the notations. A \emph{weighted graph} $\Gamma=(V,E,\alpha)$ is a graph with a weight $\alpha_{v,e}\in\mathbb C$ associated to every pair $(v,e)$ where $v$ is a vertex of $\Gamma$ and $e$ is an edge adjacent to $v$. We call such graph $\Gamma$ \emph{uniformly bounded} if the degrees of vertexes and the weights are uniformly bounded. To a uniformly bounded weighted graph $\Gamma$ we associate a \emph{Laplace type operator} on $l^2(V)$
$$(H_\Gamma f)(v)=\sum\limits_{e\in E_v}\alpha_{v,e}f(r_v(e)),\;\;\text{for}\;\; f\in l^2(V),$$ where $E_v$ stands for the set of edges adjacent to $v$ and $r_v(e)$ stands for the second vertex of an edge $e\in E_v$. For details we refer the reader to \cite{DG18 Shape}, Section 3. Let us recall the formulation of Theorem 1.6 from \cite{DG18 Shape}.
\begin{Th}[Weak Hulanicki Theorem for graphs]\label{ThGraphCov}  Let $\Gamma_1$ be a uniformly bounded connected weighted graph which covers a weighted graph $\Gamma_2$ such that either
\begin{itemize}\item[$a)$] $\Gamma_1$ is amenable and $\Gamma_2$ is finite or
\item[$b)$] $\Gamma_1$ has subexponential growth.
 \end{itemize} Let $H_1,H_2$ be the Laplace type operators associated to $\Gamma_1,\Gamma_2$. Then $\sigma(H_2)\subset \sigma(H_1)$.
\end{Th}
Let us now explain in detail the claim made in \cite{DG18 Shape}, p. 63, lines 12-14, \ie that for a uniformly bounded weighted graph $\Gamma$ the operator
$$\mathrm{I}-
\tfrac{1}{R^2}(H_\Gamma-\lambda\mathrm{I})(H_\Gamma-\lambda\mathrm{I})^{*} $$
on $l^2(V)$ coincides with an operator $H_{\Gamma'}$, where $\Gamma'=(V,E',\beta)$ for certain edge set $E'$ and a collection of weights $\beta$.  For this purpose we introduce some operations on weighted graphs.\\
{\bf Multiplication by a number.} Given a number $\lambda\in\mathbb C$ set $\lambda\Gamma=(V,E,\lambda\alpha),$ where $\lambda\alpha$ means that each weight from $\alpha$ is multiplied by $\lambda$.\\
{\bf Adding a number.} Denote by $\Gamma+\lambda$ (or $\lambda+\Gamma$) the graph obtained from $\Gamma$ by adding a loop with weight $\lambda$ at each vertex of $\Gamma$.\\
{\bf Conjugation.} Denote by $\alpha^*$ the collection of weights on the edges from $E$ such that $\alpha^*_{v,e}=\overline{\alpha_{r_v(e),e}}$ for every vertex $v\in V$ and any edge $e\in E$ adjacent to $v$. In other words, we complex conjugate each weight and change its direction on the edge. Set $\Gamma^*=(V,E,\alpha^*)$.\\
{\bf Composition of two graphs.} Let $\widetilde\Gamma=(V,\widetilde E,\widetilde\alpha)$ be another weighted graph with the same vertex set $V$. We introduce the weighted graph $\Gamma\circ\widetilde\Gamma=(V,E',\beta)$ such that $E'$ is in a bijection with the pairs of adjacent edges $(e_1,e_2),e_1\in E,e_2\in \widetilde E$. If $e_1$ joins vertexes $v$ and $w$ and $e_2$ joins vertexes $w$ and $u$ then the corresponding edge $e$ from $E'$ joins $v$ and $u$. Moreover, $\beta_{v,e}=\alpha_{v,e_1}\widetilde\alpha_{w,e_2}$.
\begin{Lm}\label{RemPropGraphs} The operations defined above satisfy the following properties:
$$H_{\lambda\Gamma}=\lambda H_\Gamma,\;\;H_{\lambda+\Gamma}=\lambda\id+H_\Gamma,\;\;
H_{\Gamma^*}=H_\Gamma^*,\;\;H_{\Gamma\circ \widetilde\Gamma}=H_\Gamma H_{\widetilde\Gamma}.$$
\end{Lm}
\begin{proof} Let us prove the last identity, the first three can be proven in a similar fashion. Recall that for two vertexes $v,w$ (not necessarily distinct) and an edge $e$ the identity $w=r_v(e)$ means that $e$ joins $v$ and $w$. By definition,
\begin{align*}(H_{\Gamma\circ \widetilde\Gamma}f)(v)=\sum\limits_{e\in E'_v}\beta_{v,e}f(r_v(e))=
\sum\limits_{e_1\in E_v}\sum\limits_{e_2\in \widetilde E_{w}}\alpha_{v,e_1}\widetilde \alpha_{w,e_2}f(r_w(e_2)),
\end{align*} where $w=r_v(e_1)$. Observing that
\begin{align*}\sum\limits_{e_2\in \widetilde E_w}\widetilde \alpha_{w,e_2}f(r_w(e_2))=(H_{\widetilde\Gamma}f)(w)\;\;\text{and}\\ \sum\limits_{e_1\in E_v}\alpha_{v,e_1} (H_{\widetilde\Gamma}f)(r_v(e_1))=
(H_\Gamma(H_{\widetilde\Gamma}f))(v)\end{align*}
we obtain the desired identity $H_{\Gamma\circ \widetilde\Gamma}=H_\Gamma H_{\widetilde\Gamma}$
\end{proof}
The following two lemmas state that the operations defined above respect graph coverings.
\begin{Lm}\label{LmCoveringProperties} Let $\Gamma_1=(V_1,E_1,\alpha_1)$, $\Gamma_2=(V_2,E_2,\alpha_2)$ be weighted graphs such that $\Gamma_1$ covers $\Gamma_2$. Then \begin{itemize}
\item{} $\lambda\Gamma_1$ covers $\lambda\Gamma_2$;
\item{} $\lambda+\Gamma_1$ covers $\lambda+\Gamma_2$;
\item{} $\Gamma_1^*$ covers $\Gamma_2^*$.
\end{itemize}\end{Lm}
The proof follows straightforwardly from definitions and we leave it to the reader as an exercise.
\begin{Lm}\label{LmCoveringProperties1}  For $i=1,2$ let $\Gamma_i=(V_i,E_i,\alpha_i)$ and $\widetilde\Gamma_i=(V_i,E_i,\widetilde\alpha_i)$ be weighted graphs with the same vertex and edge sets but possibly different weights. Assume that $\Gamma_1$ covers $\Gamma_2$ and $\widetilde\Gamma_1$ covers $\widetilde \Gamma_2$ with identical covering maps. Then
$\Gamma_1\circ\widetilde\Gamma_1$ covers
$\Gamma_2\circ\widetilde\Gamma_2.$
\end{Lm}
\begin{proof}  Let $\phi:V_1\to V_2$, $\psi:E_1\to E_2$ be the covering maps. Let $$\Gamma_1\circ\widetilde \Gamma_1=(V_1,E_1',\beta_1),\;\; \Gamma_2\circ\widetilde \Gamma_2=(V_2,E_2',\beta_2).$$ By definition of the composition of graphs, we can identify the edges of $E_i'$ with pairs of adjacent edges of $E_i$. Define $\psi':E'_1\to E'_2$ as $$\psi'((e_1,\widetilde e_1))=(\psi(e_1),\psi(\widetilde e_1)),$$ for any two adjacent edges $e_1,\widetilde e_1\in  E_1$. If $e_1$ joins vertexes $v$ and $w$ and $\widetilde e_1$ joins vertexes $w$ and $u$ from $V_1$ then by definition of a covering $\psi(e_1)$ joins $\phi(v)$ and $\phi(w)$ and $\psi(\widetilde e_1)$ joins $\phi(w)$ and $\phi(u)$. Thus, $\psi'((e_1,\widetilde e_1))\in E'_2$ joins $\phi(v)$ and $\phi(u)$. Moreover, since covering maps preserve the weights of weighted graphs, using the definition of the composition of graphs we obtain
$$\beta_{2,\phi(v),\psi'((e_1,\widetilde e_1))}=\alpha_{2,\phi(v),\psi(e_1)}\widetilde\alpha_{2,\phi(w),\psi(\widetilde e_1)}=
\alpha_{1,v,e_1}\widetilde\alpha_{1,w,\widetilde e_1}=\beta_{1,v,(e_1,\widetilde e_1)}.$$

Further, since $\psi$ is a local isomorphism for any vertex $x\in V_1$ it sends set of adjacent edges $E_{1,x}$ one-to-one onto the set of edges $E_{2,\phi(x)}$ adjacent to $\phi(x)\in V_2$. It follows that $\psi'$ is one-to-one from $E'_{1,v}$ onto $E'_{2,\phi(v)}$ for any vertex $v\in V_1$. This finishes the proof.
\end{proof}

Further, let $\Gamma_1,\Gamma_2$ be uniformly bounded weighted graphs satisfying the conditions of Theorem 1.6 from \cite{DG18 Shape}. In particular, $\Gamma_1$ covers $\Gamma_2$. Let $R$ be sufficiently large. Assume that $\lambda\in\sigma(H_{\Gamma_2})$. Then the condition $2)$ of Lemma \ref{LmEquivNorm} holds for $A=H_{\Gamma_2}$. Assume, for instance, that
$$1\in\sigma(\mathrm{I}-
\tfrac{1}{R^2}(H_{\Gamma_2}-\lambda\mathrm{I})^*(H_{\Gamma_2}-\lambda\mathrm{I}))$$ (the second case can be treated similarly). Using the operations defined above introduce the graphs
$$\Gamma_1'=1-
\tfrac{1}{R^2}(\Gamma_1-\lambda)^*\circ(\Gamma_1-\lambda),\;\;\Gamma_2'=1-
\tfrac{1}{R^2}(\Gamma_2-\lambda)^*\circ(\Gamma_2-\lambda).$$ Lemmas \ref{LmCoveringProperties} and \ref{LmCoveringProperties1} imply that $\Gamma_1'$ covers $\Gamma_2'$. Using Lemma \ref{RemPropGraphs}, we obtain:
$$H_{\Gamma_i'}=\mathrm{I}-
\tfrac{1}{R^2}(H_{\Gamma_i}-\lambda\mathrm{I})^*(H_{\Gamma_i}-\lambda\mathrm{I}),\;\;i=1,2.$$ In particular, $1\in\sigma(H_{\Gamma_2'})$.
The operators $H_{\Gamma_i'}$ are positive. From the proof of Theorem 1.6, \cite{DG18 Shape}, for the case of positive Laplace type operators  we obtain that $\sigma(H_{\Gamma_2'})\subset\sigma(H_{\Gamma_1'})$. It follows that $1\in \sigma(H_{\Gamma_1'})$. Using Lemma \ref{LmEquivNorm} we obtain that $\lambda\in \sigma(H_{\Gamma_1})$. This finishes the correction of the proof of Theorem 1.6 from \cite{DG18 Shape}.

\end{document}